\documentclass[11pt]{amsart}
\usepackage{geometry}                
\usepackage{url}
\usepackage{graphicx}
\usepackage{amssymb}
\usepackage{epstopdf}
\usepackage{color}
\usepackage{amsmath}
\usepackage{url}
\usepackage[square,numbers]{natbib}
\usepackage{subeqnarray}
\definecolor{darkred}{rgb}{.6,0,0}
\definecolor{darkblue}{rgb}{0,0,.7}

\newtheorem{lemma}{Lemma}[section]

\newtheorem{remark}{Remark}[section]

\title[Particle-based optimal control]{Particle-based algorithm for stochastic optimal control}
\author{Sebastian Reich}
\address{Institut f\"ur Mathematik, Universit\"at Potsdam, Karl-Liebknecht-Str. 24/25, 14476 Potsdam}
\date{\today}                                           

\begin{document}
\maketitle

\begin{abstract} The solution to a stochastic optimal control problem can be determined by computing the value
function from a discretization of the associated Hamilton--Jacobi--Bellman equation. Alternatively, the problem can be reformulated in terms of a 
pair of forward-backward SDEs, which makes Monte--Carlo techniques applicable. More recently, the problem has also been 
viewed from the perspective of forward and reverse time SDEs and their associated Fokker--Planck equations. This approach is closely 
related to techniques used in diffusion-based generative models. Forward and reverse time formulations express the value function as the ratio of two probability
density functions; one stemming from a forward McKean--Vlasov SDE and another one from a reverse McKean--Vlasov SDE. In this paper, we extend this approach
to a more general class of stochastic optimal control problems and combine it with ensemble Kalman filter type 
and diffusion map approximation techniques in order to obtain efficient and robust particle-based algorithms. 
\end{abstract}

%
\section{Introduction}
%

We consider controlled nonlinear diffusion processes of the form
\begin{equation} \label{eq:FSDE}
{\rm d}X_t = b(X_t){\rm d}t + G(X_t) u_t {\rm d}t + \sigma (X_t) {\rm d}B_t, \qquad X_0 = x_0.
\end{equation}
Here $X_t \in \mathbb{R}^{d_x}$ denotes the random state variable at time $t\ge 0$ and $B_t \in \mathbb{R}^{d_b}$ 
standard $d_b$-dimensional Brownian motion. Furthermore, the functions $b(x) \in \mathbb{R}^{d_x}$, $G(x) \in \mathbb{R}^{d_x \times d_u}$, and 
$\sigma (x) \in \mathbb{R}^{d_x \times d_b}$  are all assumed to be given. 

The cost to optimize via an appropriate choice of the time-dependent control $u_{0:T} = \{u_t\}_{t \in [0,T]}$ is given by
\begin{equation} \label{eq:costJ}
J_T(x_0,u_{0:T}) = \mathbb{E} \left[ \int_0^T \left(c(X_t) + \frac{1}{2} u_t^{\rm T} R^{-1} u_t \right) {\rm d}t + f(X_T) \right].
\end{equation}
Expectation is taken with regard to the path measure generated by the stochastic differential equation (SDE) (\ref{eq:FSDE}) conditioned on
the initial $X_0 = x_0$ and a given control law $u_{0:T}$, which we assume to be state-dependent. Here 
\begin{equation} \label{eq:running cost}
c(x) = \frac{1}{2} h(x)^{\rm T} S^{-1} h(x)
\end{equation}
denotes the running cost and 
\begin{equation} \label{eq:terminal cost}
f(x) = \frac{1}{2} \xi(x)^{\rm T} V^{-1} \xi(x)
\end{equation}
the terminal cost. The matrices $R \in \mathbb{R}^{d_u\times d_u}$, $S \in \mathbb{R}^{d_h \times d_h}$, $V \in 
\mathbb{R}^{d_\xi \times d_\xi}$ as well as the functions $h(x) \in \mathbb{R}^{d_h}$, $\xi(x) \in \mathbb{R}^{d_\xi}$ are again all
assumed to be given. See \cite{Pavliotis2016} for an introduction to diffusion processes and \cite{Carmona,Meyn} for an introduction to
stochastic control.

In this paper, we aim at finding control laws of the form
\begin{equation} \label{eq:affine control}
u_t (x)= R G(x)^{\rm T} (A_t x + c_t),
\end{equation}
which provide good approximations to the optimal feedback control law denoted here by $u_t^\ast(x)$. While the optimal control law can be found via
the associated Hamilton--Jacobi--Bellman (HJB) equation \cite{Carmona}; solving such a PDE is computationally demanding \cite{EHJ21}. Popular
alternatives include those based on forward-backward SDEs \cite{CKSY21} in combination with machine learning techniques \cite{EHJ17,EHJ21}. 
Here we follow the work of \cite{MO22} instead, which in turn has been
inspired by \cite{A82,reich2019data,MRO20}, to reformulate the problem in terms of two McKean--Vlasov SDEs over state space $\mathbb{R}^{d_x}$. 
Those SDEs need to be solved in forward and reverse time, respectively, only once and are related to generative models using diffusion processes
\cite{sohl2015deep,song2021scorebased}. Furthermore, in order to obtain robust and easy to compute approximations, we employ the ensemble 
Kalman filter (EnKF) methodology to approximate the arising McKean--Vlasov interaction terms \cite{Evensenetal2022,CRS22}. A related EnKF-based approach
to optimal control has been considered in \cite{JTMM22} and is based on a direct approximation of the HJB equation via a McKean--Vlasov SDE.  
We use stabilization of an inverted pendulum position \cite{Meyn} to demonstrate the efficiency of our method. While the numerical experiments in 
\cite{MO22} and \cite{JTMM22} utilize ensemble sizes on the order of $10^3$, our method has been implemented with an ensemble of size $M = d_x +1 = 3$. 
We also propose a more general methodology which combines the EnKF methodology with diffusion map approximations for the arising grad-log density terms 
\cite{diffusion_map1,diffusion_map2}. This extension is useful whenever the underlying densities cannot be approximated well by Gaussian distributions. We will illustrate this aspect through a controlled nonlinear Langevin dynamics process.

The remainder of this paper is structured as follows. The mathematical background on the HJB equation for stochastic
optimal control problems is summarized in Section \ref{sec:math}. We demonstrate in Section \ref{sec:mean field} how the value function
defined by the HJB equation can be expressed as the ratio of two probability density functions. Here we extend previous work \cite{MO22} to the
wider class of optimal control problems defined by (\ref{eq:FSDE}) and (\ref{eq:costJ}). Both the associated forward and reverse time
evolution equations can be expressed in terms of McKean--Vlasov SDEs in the state variable, $x$. Before discussing numerical approximations of those
McKean--Vlasov equations in Section \ref{sec:numerical implementation}, we demonstrate in 
Section \ref{sec:generative} how our formulation is related to generative modeling using diffusion processes \cite{sohl2015deep,song2021scorebased}.
We first discuss numerical approximations using EnKF-type methodologies \cite{Evensenetal2022,CRS22} in Section \ref{sec:EnKF}. Next
we also employ diffusion maps \cite{diffusion_map1,diffusion_map2} in order to approximate grad-log density terms in Subsection \ref{sec:DM_EnKF}.
Numerical implementation details are discussed in Subsection \ref{sec:implementation}, while numerical results 
for an inverted pendulum and nonlinear Langevin dynamics are presented in Section \ref{sec:example}.
Possible extension of the proposed methodology to infinite horizon optima control problems are discussed in Appendix \ref{sec:Appendix}.

%
\section{Mathematical problem formulation} \label{sec:math}
%

In this section, we recap the essential aspects of finding the optima control law, $u_t^\ast(x)$,
for controlled SDE (\ref{eq:FSDE}) with cost (\ref{eq:costJ}). See \cite{Carmona,Meyn} for more 
detailed expositions.

The law, $\pi_t$, of the diffusion process, $X_t$, defined by the SDE (\ref{eq:FSDE}) satisfies the Fokker--Planck equation
\begin{subequations} \label{eq:FPE}
\begin{align}
\partial_t \pi_t &= -\nabla_x \cdot \left( \pi_t \left( b + Gu_t \right) \right)+  \frac{1}{2} \nabla_x \cdot  \left(\pi_t \Sigma
\right) \\
&= -\nabla_x \cdot \left( \pi_t \left( b + Gu_t- \frac{1}{2} \nabla_x \cdot \Sigma - \frac{1}{2} \Sigma \nabla_x 
\log \pi_t \right) \right),
\end{align}
\end{subequations}
where 
\begin{equation}
\Sigma(x) = \sigma(x) \sigma(x)^{\rm T}. 
\end{equation}
We also introduce the weighted norm $\| \cdot \|_R$ via
\begin{equation}
\|u\|^2_R = u^{\rm T} R^{-1} u.
\end{equation}

Let the function $y_t(x)$ satisfy the backward HJB equation
\begin{equation} \label{eq:HJB}
-\partial_t y_t = (b + Gu_t) \cdot \nabla_x y_t + \frac{1}{2} \Sigma :  D^2_x y_t + c + \frac{1}{2} \|u_t\|^2_R
\end{equation}
with terminal condition $y_T = f$. Here $D_x^2 y_t(x) \in \mathbb{R}^{d_x \times d_x}$ denotes the Hessian of $y_t(x)$ and $A : B = \mbox{tr} \,(
A B^{\rm T})$ the Frobenius inner product of two $d_x\times d_x$ matrices $A$ and $B$. The optimal feedback control is provided by
\begin{equation} \label{eq:optimal control law}
u^\ast_t(x) = - R G(x)^{\rm T} \nabla_x y_t^\ast(x),
\end{equation}
where the optimal value function $y_t^\ast(x)$ satisfies the HJB equation
\begin{subequations}
\begin{align} 
-\partial_t y_t^\ast &= b \cdot \nabla_x y_t^\ast + \frac{1}{2} \Sigma : D^2_x y_t^\ast + c + \min_u \left( 
Gu \cdot \nabla_x y_t^\ast + \frac{1}{2} \|u\|^2_R \right)\\
&= b \cdot \nabla_x y_t^\ast + \frac{1}{2} \Sigma : D^2_x y_t^\ast + c - \frac{1}{2} \| R G^{\rm T} \nabla_x y_t^\ast \|^2_R.
\end{align}
\end{subequations}

We apply the Cole--Hopf transformation and introduce the  new function
\begin{equation}
v_t^\ast (x) := \exp(-y_t^\ast(x))
\end{equation}
in order to obtain the transformed HJB equation
\begin{subequations} \label{eq:HJB transformed}
\begin{align}
-\partial_t v_t^\ast &= b \cdot \nabla_x  v_t^\ast + \frac{1}{2} \Sigma : D_x^2  v_t^\ast  \\
& \qquad -\,
\left(c + \frac{1}{2} \|\sigma^{\rm T} \nabla_x \log v_t^\ast ||^2  
+ \min_u \left( 
\frac{1}{2} \|u\|^2_R  - Gu \cdot \nabla_x \log v_t^\ast \right) \right) v_t^\ast \\
&= b \cdot \nabla_x  v_t^\ast + \frac{1}{2} \Sigma : D_x^2  v_t^\ast  \\
& \qquad -\,
\left(c +  \frac{1}{2} \|\sigma^{\rm T} \nabla_x \log v_t^\ast ||^2  -  
\frac{1}{2} \|R G^{\rm T} \nabla_x \log v_t^\ast \|^2_R   \right) v_t^\ast .
\end{align}
\end{subequations}
The terminal condition is $v_T^\ast = \exp(-f)$. Here we have used
\begin{equation}
\partial_t y_t^\ast = \frac{-1}{v_t^\ast} \partial_t v_t^\ast, \quad
\nabla_x y_t^\ast = \frac{-1}{v_t^\ast} \nabla_x v_t^\ast, \quad \Sigma : D_x^2 y_t^\ast = \frac{-1}{v_t^\ast}
\Sigma : D_x^2 v_t^\ast + 
\|\sigma^{\rm T} \nabla_x \log v_t^\ast\|^2.
\end{equation}
The optimal control is now characterized by 
\begin{equation}
u_t^\ast(x) = RG(x)^{\rm T} \nabla_x \log v_t^\ast(x).
\end{equation}

\begin{remark}
The transformed HJB equation (\ref{eq:HJB transformed}) simplifies to
\begin{equation} 
-\partial_t v_t^\ast = b \cdot \nabla_x v_t^\ast + \frac{1}{2}  \Sigma : D_x^2 v_t^\ast  - c v_t^\ast 
\end{equation}
for the special case $G = \sigma$ and $R = I$ and 
(\ref{eq:HJB transformed}) becomes linear in $v_t^\ast$. This well-known fact has been exploited in the numerical work of \cite{MO22}. Furthermore, the transformed HJB equation arising from the further simplification $c(x)\equiv 0$  leads to the standard backward Kolmogorov equation \cite{Pavliotis2016}.
\end{remark}

%
\section{McKean--Vlasov formulation} \label{sec:mean field}
%

In this section, we extend the McKean--Vlasov forward-reverse time approach to optimal control from \cite{MO22} to more general control
problems defined by (\ref{eq:FSDE}) and (\ref{eq:costJ}). The first step is to choose an appropriate, potentially time-dependent convex cost function $\alpha_t (x)$ and to formulate a suitable evolution equation for the density
\begin{equation} \label{eq:product}
\tilde \pi_t := Z_t^{-1} v_t^\ast \bar \pi_t,
\end{equation}
where $v_t^\ast$ satisfies (\ref{eq:HJB transformed}) and $\bar \pi_t$ the forward evolution equation
\begin{equation} \label{eq:Liouville_forward}
\partial_t \bar \pi_t = - \nabla_x \cdot \left( \bar \pi_t \bar b_t \right)
 - \bar \pi_t( \alpha_t - \bar \pi_t[\alpha_t])
\end{equation}
with initial distribution $\bar \pi_0 = \delta_{x_0}$ and modified drift function
\begin{equation}
\bar b_t(x) := b(x) -\frac{1}{2}\left( \nabla_x \cdot \Sigma (x) +\Sigma(x) \nabla_x \log \bar \pi_t(x) \right).
\end{equation}
The normalization constant $Z_t$ is given by
\begin{equation}
Z_t = \bar \pi_t[v_t^\ast]
\end{equation}
and $\delta_{x_0}$ denotes the Dirac delta function centered at $x_0$. We note that
\begin{equation}
-\nabla_x \cdot \left( \bar \pi_t \bar b_t \right) = -\nabla_x \cdot (\bar \pi_t b) + \frac{1}{2} \nabla_x \cdot 
(\nabla_x \cdot ( \bar \pi_t \Sigma)).
\end{equation}

We next need to find an evolution equations for the probability density $\tilde \pi_t$ defined by (\ref{eq:product}).
Relying on
\begin{equation}
\nabla_x \log \tilde \pi_t (x) = \nabla_x \log v_t^\ast + \nabla_x \log \bar \pi_t,
\end{equation}
we introduce the modified drift
\begin{subequations}  \label{eq:tilde b}
\begin{align}
\tilde b_t(x) &:= -\bar b_t(x) - \frac{1}{2} \Sigma(x) \nabla_x \log v_t^\ast(x)\\
&= -\bar b_t(x) - \frac{1}{2} \Sigma (x) \left( \nabla_x \log \tilde \pi_t(x) + \nabla_x \log \bar \pi_t (x)\right) \\
&= -b(x) +  \nabla_x \cdot \Sigma(x) + \Sigma(x) \nabla_x \log \bar \pi_t(x) \\
&\qquad - \,\frac{1}{2} \left(\nabla_x \cdot \Sigma(x) + \Sigma(x) \nabla_x \cdot \log \tilde \pi_t (x) \right).
\end{align}
\end{subequations}
We note that
\begin{equation}
-\nabla_x \cdot (\tilde \pi_t \tilde b_t) 
= \nabla_x \cdot \left(\tilde \pi_t \left(b - \nabla_x \cdot \Sigma -\Sigma \nabla_x \log \bar \pi_t  \right) 
\right)  + \frac{1}{2}
\nabla_x \cdot (\nabla_x \cdot (\tilde \pi_t \Sigma)),
\end{equation}
which leads naturally to the interpretation in terms of a reverse time SDE with McKean--Vlasov-type drift function.
Before investigating this aspect in more detail, we state the following lemma, which links the modified
drift $\tilde b_t$ with the time evolution of the probability density $\tilde \pi_t$ defined by (\ref{eq:product}).

\begin{lemma} \label{lemma1}
Given the forward evolution equation (\ref{eq:Liouville_forward}) and the HJB equation
(\ref{eq:HJB transformed}), the probability density defined by (\ref{eq:product}) satisfies the reverse time evolution
equation
\begin{equation} \label{eq:Liouville_backward}
-\partial_t \tilde \pi_t = -\nabla_x \cdot \left(\tilde \pi_t \tilde b_t \right)
-  \tilde \pi_t \left(c - \alpha_t +  \frac{1}{2}\|\sigma^{\rm T} 
\nabla_x \log v_t^\ast \|^2 - \frac{1}{2} \|R G^{\rm T} \nabla_x \log v^\ast_t\|^2_R - \zeta_t\right)
\end{equation}
with terminal condition $\tilde \pi_T = Z_T^{-1}\exp(-f) \bar \pi_T$, where $\zeta_t$ is an appropriate normalization constant. 
\end{lemma}

\begin{proof}
Since (\ref{eq:tilde b}) and 
\begin{equation}
\nabla_x \cdot (\Sigma \nabla_x v_t^\ast) = \nabla_x v_t^\ast \cdot (\nabla_x \cdot \Sigma) + \Sigma : D_x^2 v_t^\ast,
\end{equation}
and assuming that (\ref{eq:product}) holds, 
it follows that
\begin{subequations}
\begin{align}
-\nabla_x \cdot (\tilde \pi_t \tilde b_t) 
&= \nabla_x \cdot \left(\tilde \pi_t \left( \bar b_t + \frac{1}{2} \Sigma \nabla_x \log v_t^\ast \right)\right)\\
&= \frac{v_t^\ast}{Z_t} \nabla_x  \cdot (\bar \pi_t \bar b_t ) + \frac{\bar \pi_t}{Z_t} (\Sigma \nabla_x v_t^\ast) \cdot \bar b_t
+ \frac{1}{2Z_t} \nabla_x \cdot (\bar \pi_t \Sigma \nabla_x v_t^\ast)\\
&= \frac{v_t^\ast}{Z_t} \nabla_x  \cdot (\bar \pi_t \bar b_t ) + \frac{\bar \pi_t}{Z_t} \left( \nabla_x v_t^\ast \cdot \left(
 b -\frac{1}{2} \nabla_x \cdot \Sigma   \right)
+  \frac{1}{2}  \nabla_x \cdot ( \Sigma \nabla_x v_t^\ast) \right)\\
&= \frac{v_t^\ast}{Z_t}  \nabla_x \cdot \left( \bar \pi_t \bar b_t \right)
+ \frac{\bar \pi_t}{Z_t} \left( b \cdot \nabla_x  v_t^\ast + \frac{1}{2} \Sigma : D_x^2  v_t^\ast \right).
\end{align}
\end{subequations}
Hence it holds indeed that
\begin{equation}
\partial_t \tilde \pi_t = \frac{v_t^\ast}{Z_t} \partial_t \bar \pi_t + \frac{\bar \pi_t}{Z_t} \partial_t v_t^\ast - \frac{\tilde \pi}{Z_t} \frac{{\rm d}Z_t}{{\rm d}t}
\end{equation}
for the partial time derivatives given by (\ref{eq:Liouville_backward}), (\ref{eq:Liouville_forward}), and (\ref{eq:HJB transformed}), respectively.
Furthermore, $\tilde \pi_T = v_T^\ast \bar \pi_T/Z_T$ at final time and, hence, (\ref{eq:product}) holds for all times $t\in (0,T]$.
\end{proof}

\noindent
Lemma \ref{lemma1} implies that we can solve the forward evolution equation (\ref{eq:Liouville_forward}) together with the backward evolution equation
(\ref{eq:Liouville_backward}) instead of the HJB equation (\ref{eq:HJB transformed}). Throughout the remainder of this paper, we use
\begin{equation}
\alpha_t(x) = c(x)
\end{equation}
in line with the previous work \cite{MO22}. However, other choices could be explored. See, for example, \cite{reich24}, which allows one to incorporate the terminal cost $f(x)$.

\begin{remark}
While \cite{MRO20} and \cite{MO22} form the basis for our work, we mention the alternative approach put forward in \cite{JTMM22},
where $v_t^\ast$ is viewed directly as an unnormalized probability density. This approach leads to the interpretation
of the HJB equation (\ref{eq:HJB transformed}) in terms of a nonlinear Fokker--Planck equation which in turn can be
approximated using interacting particles and EnKF-type approximations as in our work. Only deterministic
control problems are considered in \cite{JTMM22}. Furthermore, it is not obvious whether the value function 
$v_t^\ast$ is always normalizable with respect to the Lebesque measure on $\mathbb{R}^{d_x}$.
\end{remark}

\noindent
The final step is to turn (\ref{eq:Liouville_forward}) and (\ref{eq:Liouville_backward}), respectively, into forward and reverse McKean--Vlasov SDEs
\cite{MRO20}:
\begin{subequations} \label{eq:FB mean field}
\begin{align} 
{\rm d}\bar X_t &= \bar f_t^\epsilon(\bar X_t){\rm d}t + \sqrt{\epsilon}\sigma (\bar X_t) {\rm d}B_t^+,  \qquad \bar X_0 = x,\\
-{\rm d}\tilde X_t &= \tilde f_t^\epsilon(\tilde X_t){\rm d}t + \sqrt{\epsilon}\sigma (\tilde X_t){\rm d}B_t^-,  \qquad \tilde X_T \sim \tilde \pi_T.
\end{align}
\end{subequations}
Here $B_t^+$ denotes Brownian motion adapted to forward time, $B_t^-$ Brownian motion adapted to reverse time, and
$\epsilon \in (0,1]$ is a free parameter determining the noise level added to the McKean--Vlasov equations.

The drift functions are defined as follows:
\begin{equation}
\bar f_t^\epsilon (x) := b(x) - \bar g_t(x) -\frac{1-\epsilon}{2}\left( \nabla_x \cdot \Sigma (x) +\Sigma(x) \nabla_x \log \bar \pi_t(x) \right)
\end{equation}
with $\bar g_t(x)$ satisfying
\begin{equation} \label{eq:Poisson forward}
\nabla_x \cdot (\bar \pi_t \bar g_t) = -\bar \pi_t (c - \bar \pi_t[c]),
\end{equation}
and
\begin{subequations} \label{eq:tilde f}
\begin{align}
\tilde f_t^\epsilon(x) &:= 
 -b(x) + \nabla_x \cdot \Sigma(x) + \Sigma(x) \nabla_x \log \bar \pi_t(x) - \tilde g_t(x) \\
 &\qquad -\, \frac{1-\epsilon}{2} \left(
\nabla_x \cdot \Sigma(x) + \Sigma(x) \nabla_x  \log \tilde \pi_t (x) \right)
\end{align}
\end{subequations}
with $\tilde g_t(x)$ satisfying
\begin{equation} \label{eq:Poisson backward}
\nabla_x \cdot (\tilde \pi_t \tilde g_t) = -\tilde \pi_t \left(  \frac{1}{2}\|\sigma^{\rm T} 
\nabla_x \log v_t^\ast \|^2 - \frac{1}{2} \|R G^{\rm T} \nabla_x \log v^\ast_t\|^2_R - \zeta_t\right).
\end{equation}

\begin{lemma} The two diffusion processes defined by 
(\ref{eq:FB mean field}) satisfy $\bar X_t \sim \bar \pi_t$ and $\tilde X_t \sim \tilde \pi_t$, respectively. 
Given $\tilde \pi_t$ and $\bar \pi_t$, the optimal control law $u_t^\ast(x)$ is provided by
\begin{equation}
u_t^\ast(x) = R G(x)^{\rm T} \left( \nabla_x \log \tilde \pi_t(x) - \nabla_x \log \bar \pi_t(x) \right).
\end{equation}
\end{lemma}

\begin{proof}
The lemma follows immediately from writing down the associated (nonlinear) Fokker--Planck equations
for the two McKean--Vlasov SDEs (\ref{eq:FB mean field}). The stated formula for the optimal control follows
from (\ref{eq:product}). 
\end{proof}

\noindent
The choice $\epsilon = 0$ leads to fully deterministic evolution equations and the following intriguing representation of (\ref{eq:FB mean field}):
\begin{subequations} \label{eq:determinstic McKean Vlasov}
\begin{align}
\frac{{\rm d}\bar X_t}{{\rm d}t} &=
b(\bar X_t) - \bar g_t(\bar X_t) 
-\frac{1}{2}\left( \nabla_x \cdot \Sigma (\bar X_t) +\Sigma(\bar X_t) \nabla_x \log \bar \pi_t(\bar X_t) \right)\\
-\frac{{\rm d}\tilde X_t}{{\rm d}t}&=
 - \tilde g_t(\tilde X_t) - \bar g_t(\tilde X_t)
 -  \frac{{\rm d}\bar X_t}{{\rm d}t}(\tilde X_t) 
 +\frac{1}{2}\Sigma(\tilde X_t)  
\nabla_x y_t^\ast (\tilde X_t)  .
\end{align}
\end{subequations}
We will return to this formulation in Appendix \ref{sec:Appendix}.

%
\section{A brief diversion: Diffusion-based generative modeling} \label{sec:generative}
%

Before discussing numerical implementation of the proposed forward-reverse McKean--Vlasov equations (\ref{eq:FB mean field}), 
we demonstrate how the core idea of diffusion-based generative modeling \cite{sohl2015deep,song2021scorebased}, arises as a special instance of  (\ref{eq:FB mean field}). 
Let us start from the control SDE formulation
\begin{equation} \label{eq:DM1}
{\rm d}X_t = -\frac{1}{2} X_t {\rm d}t + u_t{\rm d}t + {\rm d}B_t
\end{equation}
with initial conditions $X_0 \sim \pi_0 = {\rm N}(0,I)$. The cost function (\ref{eq:costJ})
is implicitly given by 
\begin{equation}
e^{-f(x)} \propto \frac{\pi_{\rm data}(x)}{\pi_0(x)}
\end{equation}
with running cost $c=0$ and $R = I$. Here $\pi_{\rm data}$ denotes the data distribution. We also note
that $\pi_0$ is the invariant distribution of (\ref{eq:DM1}) for $u_t = 0$ and that our forward
SDE (\ref{eq:FB mean field}a) simply becomes
\begin{equation}
{\rm d}\bar X_t = -\left(\frac{1}{2}\bar X_t - \frac{1-\epsilon}{2}\bar C_t^{-1}(\bar X_t-\bar m_t)\right)  {\rm d}t + \sqrt{\epsilon}{\rm d}B_t^+,  \qquad \bar X_0 \sim {\rm N}(0,I).
\end{equation}
with $\bar C_t = I$ and $\bar m_t = 0$ for all $t\ge 0$. Furthermore, setting $\epsilon = 0$, leads to
\begin{equation}
\frac{{\rm d}\bar X_t}{{\rm d}t} = 0.
\end{equation}
Similarly, the reverse SDE (\ref{eq:FB mean field}b) reduces to
\begin{equation}
-{\rm d}\tilde X_t = -\left(\frac{1}{2}\tilde X_t  + \frac{1-\epsilon}{2} \nabla_x \log \tilde \pi_t(\tilde X_t) \right) 
{\rm d}t + \sqrt{\epsilon}{\rm d}B_t^-,  \qquad \tilde X_T \sim \tilde \pi_{\rm data},
\end{equation}
which is of the standard form used in diffusion modeling for $\epsilon = 1$. The desired control term is finally provided by
\begin{equation}
u_t^\ast(x) = \nabla_x \log \tilde \pi_t(x) + x
\end{equation}
and the generative SDE model (\ref{eq:DM1}) turns into
\begin{equation} \label{eq:DM2}
{\rm d}X_t = \frac{1}{2} X_t {\rm d}t  + \nabla_x \log \tilde \pi_t (X_t) {\rm d}t - \frac{1-\epsilon}{2} \nabla_x \log \pi_t(X_t){\rm d}t + \sqrt{\epsilon}{\rm d}B_t, \quad
X_0 \sim {\rm N}(0,I),
\end{equation}
which again reduces to the standard diffusion-based generative model for $\epsilon = 1$. 
A more detailed discussion on the connection between diffusion-based generative modeling and stochastic optimal control can be found in \cite{BRU22}.

%
\section{Numerical implementations} \label{sec:numerical implementation}
%

In this section, we discuss the numerical implementation of the proposed forward-reverse McKean--Vlasov SDEs
(\ref{eq:FB mean field}). We start with Gaussian and EnKF-type approximations \cite{Evensenetal2022,CRS22} before also utilizing diffusion maps 
\cite{diffusion_map1,diffusion_map2} in order to approximate the required grad-log density terms.

%
\subsection{EnKF approximation} \label{sec:EnKF}
%

We develop a numerical implementation of (\ref{eq:FB mean field}) based on the EnKF methodology 
\cite{Evensenetal2022,CRS22}. In particular, we approximate a drift $g_t$, which should satisfy
\begin{equation}
\nabla_x \cdot (\pi_t g_t)  = -\pi_t (\|\psi\|_B^2 - \pi_t[\|\psi\|_B^2])
\end{equation}
for given function $\psi(x)$ and density $\pi_t$,
in the following manner.  We introduce the mean, $m_t^\psi$, of $\psi(x)$ and the 
covariance matrix, $C_t^{x\psi}$, between $x$ and $\psi(x)$ under $\pi_t$. Then an approximative drift term $g_t^{\rm KF}$ is defined by
\begin{equation} \label{eq:bit}
g_t^{\rm KF}(x) := \frac{1}{2} C_t^{x\psi} B^{-1} \left(\psi(x) + m^\psi_t\right).
\end{equation}
This approximation becomes exact for Gaussian density $\pi_t$ and linear function $\psi(x)$. See, for example, \cite{reich2011dynamical,reich2019data}
for the general methodology and \cite{JTMM22} for an application to optimal control. 

We assume that the running cost $c(x)$ is of the form (\ref{eq:running cost}).
Hence, following (\ref{eq:Poisson forward}) and (\ref{eq:bit}), the drift term $\bar g_t$ is approximated by
\begin{equation}
\bar g_t^{\rm KF}(x) := \frac{1}{2} \bar C_t^{xh} S^{-1} \left( h(x) + \bar m_t^h \right).
\end{equation}
Here $\bar C_t^{xh}$ denotes the covariance matrix between $x$ and $h(x)$ with respect to $\bar \pi_t$ and
$\bar m_t^h$ the mean of $h(x)$ under the same distribution. 

The transformation from $\bar X_T$ to $\tilde X_T$ under the terminal cost (\ref{eq:terminal cost})
is performed by the stochastic EnKF \cite{Evensenetal2022,CRS22}; that is,
\begin{equation} \label{eq:EnKF final time}
\tilde X_T = \bar X_T - \bar C_T^{x\xi} \left(\bar C_T^{\xi \xi} + V\right)^{-1}(\xi(\bar X_T) + V^{1/2} \Xi), \qquad \Xi \sim {\rm N}(0,I).
\end{equation}

Following the desired control {\it ansatz} (\ref{eq:affine control}), we also approximate $\nabla_x \log v_t^\ast$ as a linear function using 
the first two moments of $\tilde \pi_t$ and $\bar \pi_t$, respectively; that is,
\begin{subequations}
\begin{align}
\nabla_x \log v_t^{\rm KF} (x) &:= \bar C_t^{-1}(x-\bar m_t^x) -\tilde C_t^{-1}(x-\tilde m_t^x) \\
&= (\bar C_t^{-1} - \tilde C_t^{-1} )x + \left( \tilde C_t^{-1} \tilde m_t^x
- \bar C_t^{-1} \bar m_t^x\right) = A_t x + c_t
\end{align}
\end{subequations}
with
\begin{equation} \label{eq:Ac}
A_t := \bar C_t^{-1} - \tilde C_t^{-1}, \qquad c_t := \tilde C_t^{-1} \tilde m_t^x
- \bar C_t^{-1} \bar m_t^x.
\end{equation}
This approximation leads to the further approximation 
\begin{equation} \label{eq:EnKF tilde g}
\tilde g_t^{\rm KF}(x) := \frac{1}{2} \tilde C_t A_t \left(\Sigma(\tilde m_t^x)-G(\tilde m_t^x)RG(\tilde m_t^x)^{\rm T} \right) \left(A_tx + A_t\tilde m_t^x + 2c_t \right)
\end{equation}
for the drift term $\tilde g_t$ arising from (\ref{eq:Poisson backward}). 

We now summarize our approximations to the drift terms in the forward-reverse McKean--Vlasov equations (\ref{eq:FB mean field}):
\begin{equation} \label{eq:EnKF_forward}
\bar f_t^\epsilon (x) := b(x) -\frac{1-\epsilon}{2}\left( \nabla_x \cdot \Sigma (x) -\Sigma(x) \bar C_t^{-1} (x - \bar m_t^x) \right) 
-\frac{1}{2} \bar C_t^{xh} S^{-1} \left( h(x) + \bar m_t^h \right)
\end{equation}
and
\begin{subequations} \label{eq:EnKF_backward}
\begin{align}
\tilde f_t^\epsilon(x) &:= 
 -b(x) + \nabla_x \cdot \Sigma(x) - \Sigma(x) \bar C_t^{-1}( x - \bar m_t^x)  \\
 &\qquad -\, \frac{1-\epsilon}{2} \left(
\nabla_x \cdot \Sigma(x) - \Sigma(x) \tilde C_t^{-1} (x - \tilde m_t^x) \right) \\
&\qquad -\,\frac{1}{2} \tilde C_t A_t \left(\Sigma(\tilde m_t^x)-G(\tilde m_t^x)RG(\tilde m_t^x)^{\rm T} \right) \left(A_tx + A_t\tilde m_t^x + 2c_t \right).
\end{align}
\end{subequations}
The forward equation (\ref{eq:FB mean field}a) is solved from the initial condition $\bar X_0 = x_0$. The terminal $\bar X_T$ is transformed
into the terminal condition $\tilde X_T$ using (\ref{eq:EnKF final time}). Equation (\ref{eq:FB mean field}b) is solved 
from $t = T$ to $t=0$. The desired approximation to the optimal control $u_t^\ast$ is provided by (\ref{eq:affine control}) with $A_t$ and $c_t$ given by
(\ref{eq:Ac}). We note that $A_t$ is symmetric negative-definite whenever $\tilde C_t \bar C_t^{-1} < I$. In other words, the covariance matrix of the reverse process
has to be strictly smaller than the covariance matrix of the forward process in order for the associated control (\ref{eq:affine control}) to act in a stabilizing manner.

We also note that the McKean--Vlasov contribution (\ref{eq:EnKF_backward}c) stabilizes the reverse dynamics provided
\begin{equation}
\Sigma(x) > G(x)RG(x)^{\rm T} 
\end{equation}
and destabilizes it otherwise. The overall reverse dynamics can still be stable due to the contributions from (\ref{eq:EnKF_backward}a).

%
\subsection{Combined diffusion map and EnKF approximation} \label{sec:DM_EnKF}
%

In this subsection, we propose another implementation of the McKean--Vlasov formulation
(\ref{eq:FB mean field}) combining the EnKF-type approximations for the drift terms
$\bar g_t(x)$ and $\tilde g_t(x)$, respectively, while using diffusion maps \cite{diffusion_map1,diffusion_map2}
for estimating grad-log density terms.

We first consider the forward McKean--Vlasov equations
\begin{equation} \label{eq:forward pure SDE}
{\rm d}\bar X_t = b(\bar X_t) {\rm d}t - \frac{1}{2} \bar C_t^{xh} \left( h(\bar X_t) + \bar m_t^h \right){\rm d}t + \sigma(\bar X_t){\rm d}B_t^{+}.
\end{equation}
The law $\bar \pi_t$ of $\bar X_t$ induces the generator $\mathcal{\bar L}_t$ at time $t$, which is 
defined by
\begin{equation}
\mathcal{\bar L}_t f := \frac{1}{\bar \pi_t} \nabla_x \cdot (\bar \pi_t \Sigma \nabla_x f).
\end{equation}
Here we assume that $\Sigma (x)$ has full rank. It is easy to verify that
\begin{equation} \label{eq:generator}
\mathcal{\bar L}_t {\rm Id} = \nabla_x \cdot \Sigma + \Sigma \nabla_x \log \bar \pi_t,
\end{equation}
where ${\rm Id}:\mathbb{R}^{d_x} \to \mathbb{R}^{d_x}$ denotes the identity map; that is, ${\rm Id}(x) = x$.
Equation (\ref{eq:generator}) suggests the approximation
\begin{equation} \label{eq:DM approximation}
\nabla_x \cdot \Sigma(x) + \Sigma(x) \nabla_x \log \bar \pi_t(x) \approx \frac{\exp(\varepsilon \mathcal{\bar L}_t) {\rm Id} - {\rm Id}}{\varepsilon}(x)
\end{equation}
for $\varepsilon >0$ sufficiently small, where the semi-group $\exp(\varepsilon \mathcal{L}_t)$ will be later replaced by the normalized 
diffusion map approximation as investigated in \cite{WR20}. We introduce the conditional mean
\begin{equation}
\bar m_t^\varepsilon(x) := \exp(\varepsilon \mathcal{\bar L}_t){\rm Id}(x)
\end{equation}
and obtain the compact representation
\begin{equation}
 \frac{\exp(\varepsilon \mathcal{\bar L}_t) {\rm Id} - {\rm Id}}{\varepsilon}(x) = \varepsilon^{-1} (\bar m_t^\varepsilon(x)-x).
 \end{equation}

Approximation (\ref{eq:DM approximation}) is plugged into the reverse McKean--Vlasov equation to yield
\begin{subequations} \label{eq:reverse time SDE}
\begin{align}
-{\rm d}\tilde X_t &= -b(\tilde X_t){\rm d}t 
+ \varepsilon^{-1}(\bar m_t^\varepsilon(\tilde X_t) - \tilde X_t){\rm d}t 
- \tilde g_t^{\rm KF}(\tilde X_t){\rm d}t \\
& \qquad -\, \frac{1-\epsilon}{2} \left(
\nabla_x \cdot \Sigma(\tilde X_t) - \Sigma(\tilde X_t) \tilde C_t^{-1} (\tilde X_t - \tilde m_t^x) \right) + \sqrt{\epsilon} \sigma(\tilde X_t) {\rm d}B_t^{-},
\end{align}
\end{subequations}
where $\tilde g_t^{\rm KF}(x)$ is defined by (\ref{eq:EnKF tilde g}) as before and $\epsilon \in [0,1]$. 
 
Approximation (\ref{eq:DM approximation}) can also be used in the forward McKean--Vlasov equation and (\ref{eq:forward pure SDE}) gets replaced by
\begin{subequations} \label{eq:forward time SDE}
\begin{align}
{\rm d}\bar X_t &= b(\bar X_t) {\rm d}t - \frac{1}{2} \bar C_t^{xh} \left( h(\bar X_t) + \bar m_t^h \right){\rm d}t \\
& \qquad- \frac{1-\epsilon}{2\varepsilon} (\bar m_t^\varepsilon(\bar  X_t) - \bar X_t) {\rm d}t 
+ \sqrt{\epsilon}\sigma(\bar X_t){\rm d}B_t^{+}.
\end{align}
\end{subequations}

Furthermore, the law $\tilde \pi_t$ of $\tilde X_t$ induces the generator $\mathcal{\tilde L}_t$ at time $t$, which is 
defined by
\begin{equation}
\mathcal{\tilde L}_t f := \frac{1}{\tilde \pi_t} \nabla_x \cdot (\tilde \pi_t \Sigma \nabla_x f),
\end{equation}
and which can be used to approximate 
\begin{equation} 
\nabla_x \cdot \Sigma(x) + \Sigma(x) \nabla_x \log \tilde \pi_t(x) \approx \varepsilon^{-1}(\tilde m_t^\varepsilon(x)
-x), \quad
\tilde m_t^\varepsilon(x) := \exp(\varepsilon \mathcal{\tilde L}_t) {\rm Id} (x),
\end{equation}
in the reverse SDE drift function (\ref{eq:tilde f}b). In other words, (\ref{eq:reverse time SDE}) gets replaced by
\begin{subequations} \label{eq:reverse time SDE II}
\begin{align}
-{\rm d}\tilde X_t &= -b(\tilde X_t){\rm d}t 
+ \varepsilon^{-1}(\bar m_t^\varepsilon(\tilde X_t) - \tilde X_t){\rm d}t 
- \tilde g_t^{\rm KF}(\tilde X_t){\rm d}t \\
& \qquad -\, \frac{1-\epsilon}{2\varepsilon}(\tilde m_t^\varepsilon(\tilde X_t)-\tilde X_t)  + \sqrt{\epsilon} \sigma(\tilde X_t) {\rm d}B_t^{-}.
\end{align}
\end{subequations}
We note that the McKean--Vlasov equations 
(\ref{eq:forward time SDE}) and (\ref{eq:reverse time SDE II}) become deterministic under the choice $\epsilon = 0$.

%
\subsection{Numerical implementation details} \label{sec:implementation}
%

We numerically implement the McKean--Vlasov equations (\ref{eq:FB mean field}) with drift terms given by
(\ref{eq:EnKF_forward}) and (\ref{eq:EnKF_backward}) using a Monte Carlo approach; that is, we propagate
an ensemble of $M$ particles $\bar X_t^{(i)}$, $i=1,\ldots,M$, forward in time,  $t\in [0,T]$, and an equally sized ensemble of particles 
$\tilde X_t^{(i)}$ backward in time. The required mean values and covariance matrices are replaced by
their empirical estimators. A covariance inflation of $\delta I$, $\delta >0$, is added to the empirical covariance matrices in order to ensure that they remain non-singular
\cite{Evensenetal2022}. For the purpose of this paper, we apply a simple Euler--Maruyama time-stepping
method with step-size $\Delta t$ both in forward and reverse time \cite{kloeden1991numerical}. More robust time-stepping methods can be based on the 
formulations proposed and investigated in \cite{amezcuaensemble}.

When running the EnKF-type formulation from Subsection \ref{sec:EnKF},
we set the initial conditions to $\bar X_0^{(i)} = x_0$ in the forward equation and use $\epsilon > 0$ for the first time-step in order to diffuse these identical 
particles. All subsequent time-steps employ then $\epsilon = 0$ (deterministic dynamics). 
The terminal ensemble $\tilde X_T^{(i)}$, $i=1,\ldots,M$,  is computed using the forward ensemble $\bar X_T^{(i)}$ at final time and 
a standard ensemble implementation of the EnKF update (\ref{eq:EnKF final time}) \cite{Evensenetal2022,CRS22}.
The reverse McKean--Vlasov equations are solved with $\epsilon = 0$ (deterministic dynamics). 

The reverse McKean--Vlasov equation (\ref{eq:reverse time SDE}) also requires the approximation of the semi-group $\exp(\epsilon \mathcal{\bar L}_t)$. 
We now describe an implementation which follows ideas from \cite{GLMR23}. Based on the forward-in-time samples $\{\bar X_t^{(i)}\}$, we first define the diffusion map 
approximation
\begin{equation} \label{eq:DM}
P_t^\varepsilon = D(v_t^\varepsilon) R_t^\varepsilon D(v_t^\varepsilon),
\end{equation}
where the matrix $R_t^\varepsilon \in \mathbb{R}^{M\times M}$ has entries
\begin{equation}
(R_t^\varepsilon)_{ij} = \exp \left(\frac{-1}{2\varepsilon} (\bar X_t^{(i)}- \bar X_t^{(j)})^{\rm T} \left( \Sigma(\bar X_t^{(i)}) + \Sigma(\bar X_t^{(j)}) \right)^{-1}
(\bar X_t^{(i)}- \bar X_t^{(j)}) \right),
\end{equation}
$D(v) \in \mathbb{R}^{M\times M}$ denotes the diagonal matrix with diagonal entries given by the vector $v\in \mathbb{R}^M$, and 
the vector $v_t^\varepsilon \in \mathbb{R}^M_+$ is
chosen such that
\begin{equation}
\sum_{i=1}^M (P_t^\varepsilon)_{ij} = \sum_{j=1}^M (P_t^\varepsilon)_{ij} = \frac{1}{M}.
\end{equation}
The vector $v_t^\varepsilon$ can be computed efficiently using the iterative algorithm from \cite{WR20}.

We define a probability vector $p_t^\varepsilon (x) \in \mathbb{R}^M$ for all $x \in \mathbb{R}^{d_x}$ as follows. First we introduce
the vector $r_t^\epsilon (x) \in \mathbb{R}^M$ with entries
\begin{equation}
(r_t^\varepsilon)_i (x) = \exp \left(\frac{-1}{2\varepsilon} (\bar X_t^{(i)}- x)^{\rm T} \left( \Sigma(\bar X_t^{(i)}) + \Sigma(x) \right)^{-1}
(\bar X_t^{(i)}- x) \right)
\end{equation}
for $i=1,\ldots,M$.
Next we compute the vector $v_t^\varepsilon$ in (\ref{eq:DM}), which in turn is used to define
\begin{equation}
p_t^\varepsilon(x) = \frac{D(v_t^\varepsilon) r_t^\varepsilon (x)}{(v_t^\varepsilon)^{\rm T} r_t^\epsilon (x)}.
\end{equation}
Setting $\epsilon = \Delta t$, we finally obtain the approximation
\begin{equation} \label{eq:DM app}
\bar m_t^{\Delta t}(x) = \exp(\Delta t \mathcal{\bar L}_t){\rm Id}(x) \approx \bar{\mathcal{X}}_t p_t^{\Delta t} (x)
\end{equation}
with
\begin{equation}
\bar{\mathcal{X}}_t  = \left( \bar X_t^{(1)},\bar X_t^{(2)}, \ldots,\bar X_t^{(M)}\right) \in \mathbb{R}^{d_x\times M}.
\end{equation}
The reverse McKean--Vlasov equation (\ref{eq:reverse time SDE}), here with $\epsilon = 1$ for simplicity, is integrated backward in time using the following split-step scheme:
\begin{subequations}
\begin{align}
\tilde X_{t_{n-1/2}}^{(i)} &= \tilde X_{t_n}^{(i)} - \Delta t b(\tilde X_{t_n}^{(i)}) - \Delta t \tilde g_{t_n}^{\rm KF}(\tilde X_{t_n}^{(i)}){\rm d}t + \sqrt{\Delta t} \sigma(\tilde X_{t_n}
^{(i)}) \Xi_{t_n}^{(i)},\\
\tilde X_{t_{n-1}}^{(i)} &= \bar{\mathcal{X}}_{t_{n-1}}p_{t_{n-1}}^{\Delta t} (\tilde X_{t_{n-1/2}}^{(i)}),
\end{align}
\end{subequations}
$i=1,\ldots,M$, where $\Xi_{t_n}^{(i)}$ denote independent standard Gaussian random variables with mean zero and identity covariance matrix, and $t_{n+1} = t_n
+\Delta t$. This implementation guarantees that any reverse time solution $\tilde X_{t_n}$ is contained in the convex hull generated by the
forward samples $\{\bar X_{t_n}^{(i)}\}$ \cite{GLMR23}, which is a desirable property in terms of $\tilde \pi_t \propto v_t^\ast \bar \pi_t \ll \bar \pi_t$. 
The approximation (\ref{eq:DM app}) can also be applied in the forward McKean--Vlasov equation (\ref{eq:forward time SDE}) in case $\epsilon < 1$.

Please note that we propose to still approximate the optimal control $u^\ast_t(x)$ by (\ref{eq:affine control}) with $A_t$ and 
$c_t$ given by (\ref{eq:Ac}). However, diffusion map approximations could also be used in this context. See also \cite{MRO20}.

It should be noted that the diffusion map approximation requires $M \gg d_x$, which is in contrast to the EnKF-type approximation from 
Subsection \ref{sec:EnKF}, which can be implemented with as little as $M = d_x  + 1$ particles in order to render the empirical covariance matrices 
non-singular and, hence, to obtain well-defined evolution equations at the particle level. This desirable property is verified in the following section.  
However, EnKF-type approximations can fail due to stability and accuracy reasons and need to then be  
augmented by diffusion map approximations as we also demonstrate in the following section.

%
\section{Numerical examples} \label{sec:example}
%

In this section, we discuss numerical findings for two simple control problems. The first control problems is to stabilise the unstable equilibrium position
of a mathematical pendulum. This control problem is nonlinear in nature and linear feedback control laws will be suboptimal. However, we find that
(\ref{eq:affine control}) is nevertheless able to drive the pendulum from the stable to the unstable equilibrium in finite time. The second control problem
concerns the stabilization of an unstable equilibrium point of one-dimensional nonlinear Langevin dynamics. Here the computational challenge
arises from the fact that the drift term $b(x)$ becomes strongly destabilizing when integrated backward in time which requires a diffusion map approximation
of the stabilizing grad-log density term in the reverse McKean--Vlasov dynamics. 

%
\subsection{Inverted pendulum}
%

As a first example, we consider the inverted pendulum with control \cite{Meyn}. The state variable is $x = (\theta,\dot{\theta})^{\rm T} \in \mathbb{R}^2$ with
equations of motion
\begin{subequations} \label{eq:example}
\begin{align}
{\rm d} \theta &= \dot{\theta}{\rm d}t,\\
{\rm d} \dot{\theta} &= \sin (\theta){\rm d}t - \cos(\theta)u {\rm d}t + \rho {\rm d}B_t,
\end{align}
\end{subequations}
and $\rho = 1$. Consider the running cost  
\begin{equation}
c(x) = \frac{10}{2} \|\dot{\theta}\|^2
\end{equation}
over a finite time window $t\in [0,1]$ with final cost 
\begin{equation} \label{eq:final cost}
f(x) = \frac{10^3}{2} \|x\|^2.
\end{equation}
The control is scalar-valued and the penalty term in the cost function uses $R = 10$. Note that $G(x)$ is position dependent and
that $(0,0)$ is an unstable equilibrium point of the deterministic pendulum ($\rho = 0$). The noise acts only on the momentum equation.
We seek a control law of the form (\ref{eq:affine control}) 
that leads us from the stable equilibrium $(\pi,0)$ to the unstable one $(0,0$) at time $T=1$. We
initialize $\bar X_0 = (\pi,0.1)$; that is, we give the stable equilibrium a small initial kick.

The numerical experiment uses $M=3$ ensemble members in the EnKF-type formulation from
Subsection \ref{sec:EnKF}. The time-step is
set to $\Delta t = 10^{-4}$ and $\epsilon = 0.01$ for the first time-step of the forward dynamics. The additive covariance inflation factor is
set to $\delta = 10^{-4}$. The results from the forward and reverse McKean--Vlasov equations can be found in Figures \ref{fig1}.
It can be seen that the forward dynamics stays close to the stable equilibrium point $(\pi,0)$ over the whole time interval $[0,1]$. 
The stiff final cost implied by (\ref{eq:final cost}) transforms the ensemble $\bar X_T^{(i)}$ to an ensemble $\tilde X_T^{(i)}$, $i=1,\ldots,M$, which is tightly 
clustered about the unstable equilibrium $(0,0)$ at $T=1$. Solving the reverse McKean--Vlasov equations leads us gradually back to the unstable equilibrium, 
which is reached at time $t=0$.

We then apply the computed control (\ref{eq:affine control}) to the inverted pendulum equations (\ref{eq:example}) 
with the noise set to zero ($\rho = 0)$. The time evolution of the resulting solution is displayed in Figure \ref{fig2}. 
The computed time-dependent affine control is able to drive the solution from the stable to the unstable equilibrium point over a unit time interval.
The time evolution of the associated velocity indicates that strong acceleration terms are required and indeed provided by the
computed control law.

\begin{figure}[!htb]
	\begin{center}
	\includegraphics[width=0.45\textwidth]{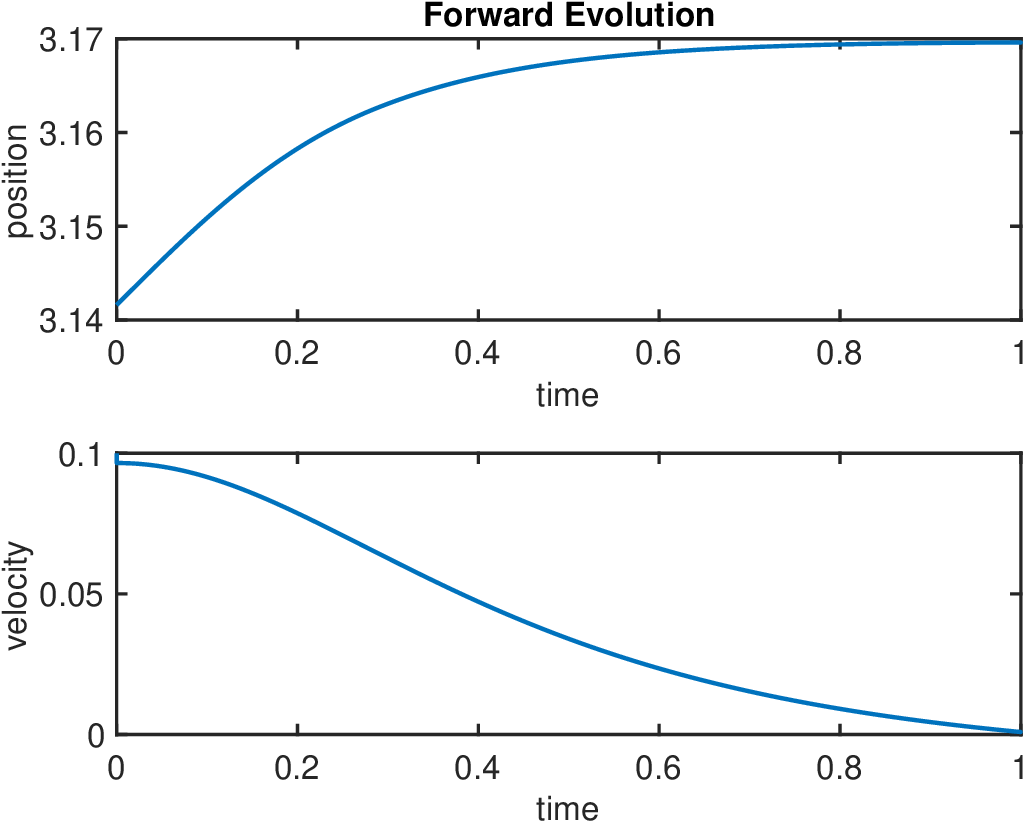} $\qquad$
	\includegraphics[width=0.45\textwidth]{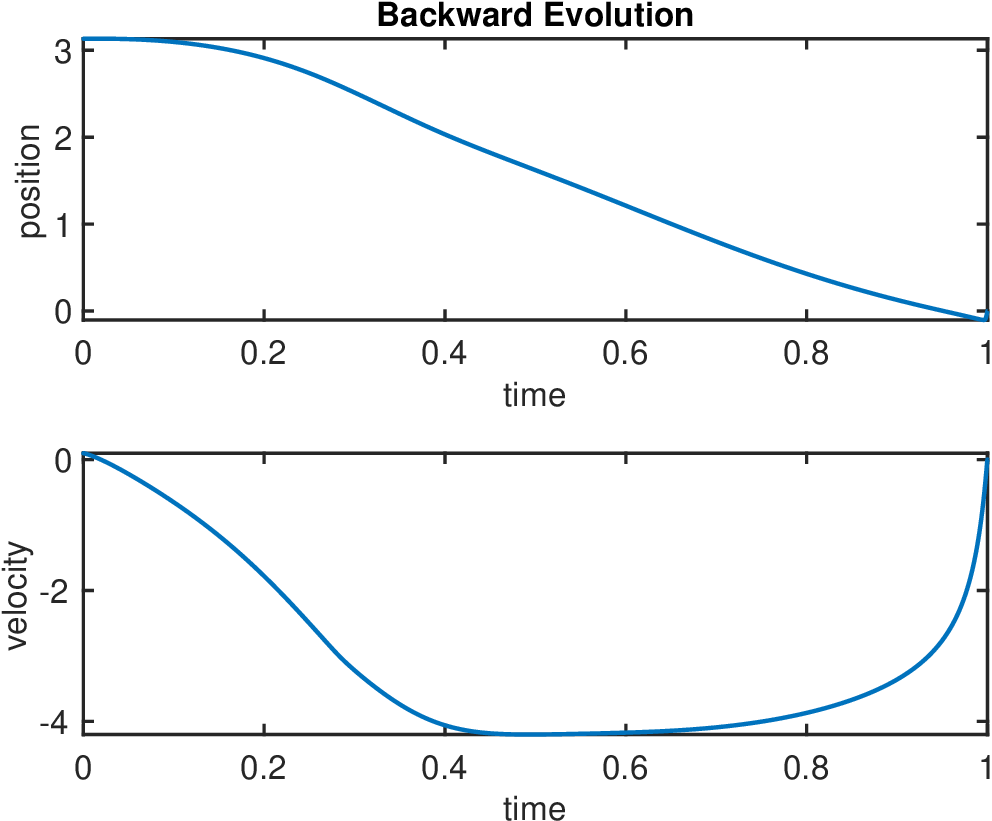}
	\end{center}
	\caption{Time evolution of the ensemble mean from the forward evolution (left panel) and the reverse evolution (right panel) both in terms
	of pendulum position and velocity. It can be seen that the reverse evolution connects the stable and unstable equilibrium points while the forward 
	dynamics stays close to the stable equilibrium.} \label{fig1}
\end{figure}

\begin{figure}[!htb]
	\begin{center}
	\includegraphics[width=0.45\textwidth]{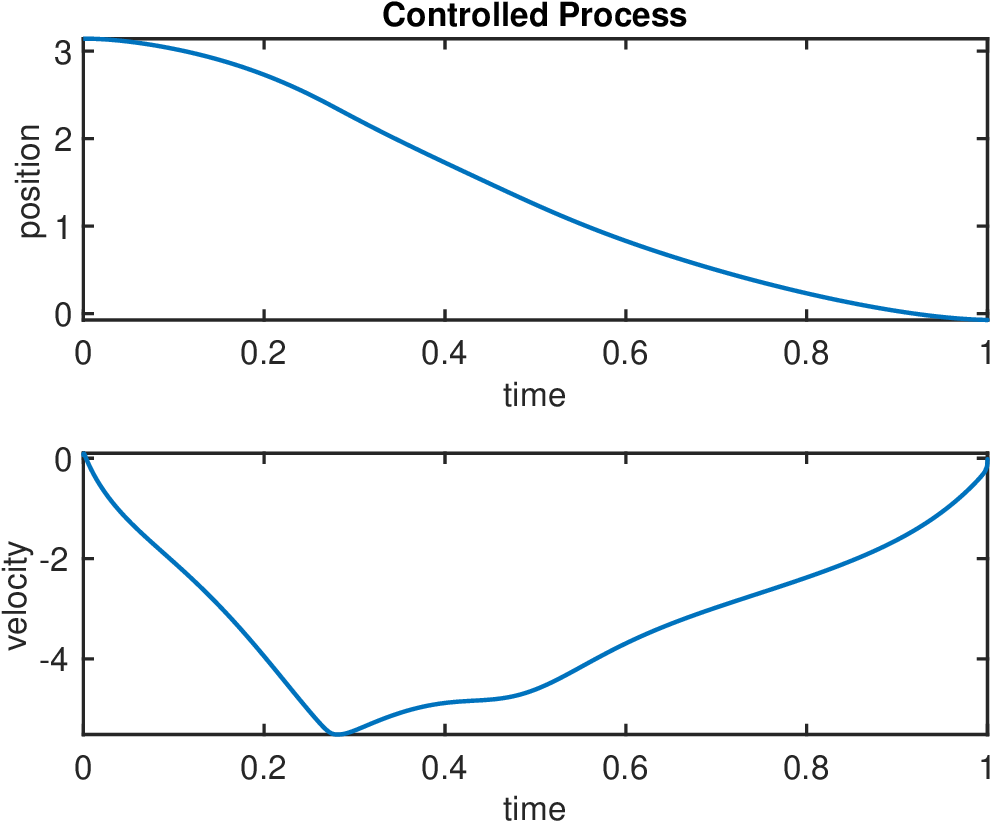} 
	\end{center}
	\caption{Time evolution of the position and velocity of the pendulum under the computed affine control law. The pendulum leaves its initial stable solution
	to reach the unstable equilibrium at time $T=1$. The initial and final velocities are essentially zero.} \label{fig2}
\end{figure}

%
\subsection{Controlled Langevin dynamics} \label{sec:CLD}
%

As a second example, we consider the controlled Langevin dynamics
\begin{equation} \label{eq:LD}
{\rm d}X_t = - (X_t^3 - X_t){\rm d}t + u_t(X_t){\rm d}t + {\rm d}B_t
\end{equation}
with unstable equilibrium at $x = 0$ and two stable equilibria at $x = \pm 1$. The imposed running cost is
$c(x) = 100 x^2/2$ and the terminal cost at $T=30$ is $f(x) = x^2/2$. 
We implement the combined diffusion map and EnKF scheme from Subsection \ref{sec:DM_EnKF}
with step-size $\Delta t = 0.01$, $M = 8$ ensemble members, and $\epsilon = 0$ in the forward  (\ref{eq:forward time SDE}) and reverse
(\ref{eq:reverse time SDE}) dynamics except for the first ten steps of the forward dynamics (\ref{eq:forward time SDE}) where we set $\epsilon = 1$. 
We also employ additive ensemble inflation with $\delta = 10^{-4}$. 

The diffusion map approximation of the grad-log density term in the reverse dynamics is essential for counterbalancing
the strongly unstable contribution stemming from the drift term in (\ref{eq:LD}) when integrated backward in time. We also find that 
the Gaussian approximation to the grad-log density term in the forward dynamics is insufficient and that
the diffusion map approximation in (\ref{eq:forward time SDE}) significantly improves the behaviour of the deterministic formulation 
($\epsilon = 0$). The scale parameter in the diffusion map approximation is set to $\varepsilon = \Delta t$. 

Except for brief transition periods at initial and final time, the control law (\ref{eq:affine control}) is essentially time-independent. See Figure \ref{fig3} and
the Appendix \ref{sec:Appendix} for a related discussion of infinite horizon optimal control problems. The effectiveness of the control is demonstrated in
Figure \ref{fig4}.

\begin{figure}[!htb]
	\begin{center}
	\includegraphics[width=0.45\textwidth]{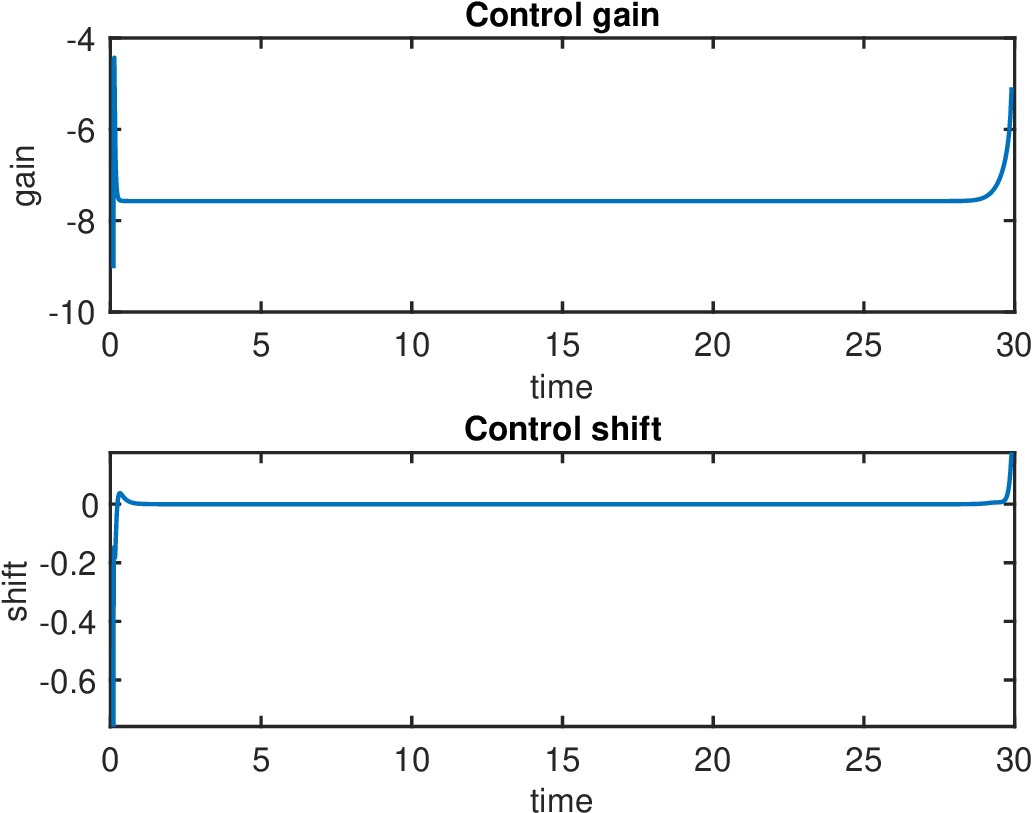} 
	\end{center}
	\caption{Computed control gain $A_t$ and shift $c_t$ from the forward and reverse McKean--Vlasov evolution equations. The control is time-independent except
	for brief transition periods at the beginning and end of the simulation interval.} \label{fig3}
\end{figure}

\begin{figure}[!htb]
	\begin{center}
	\includegraphics[width=0.45\textwidth]{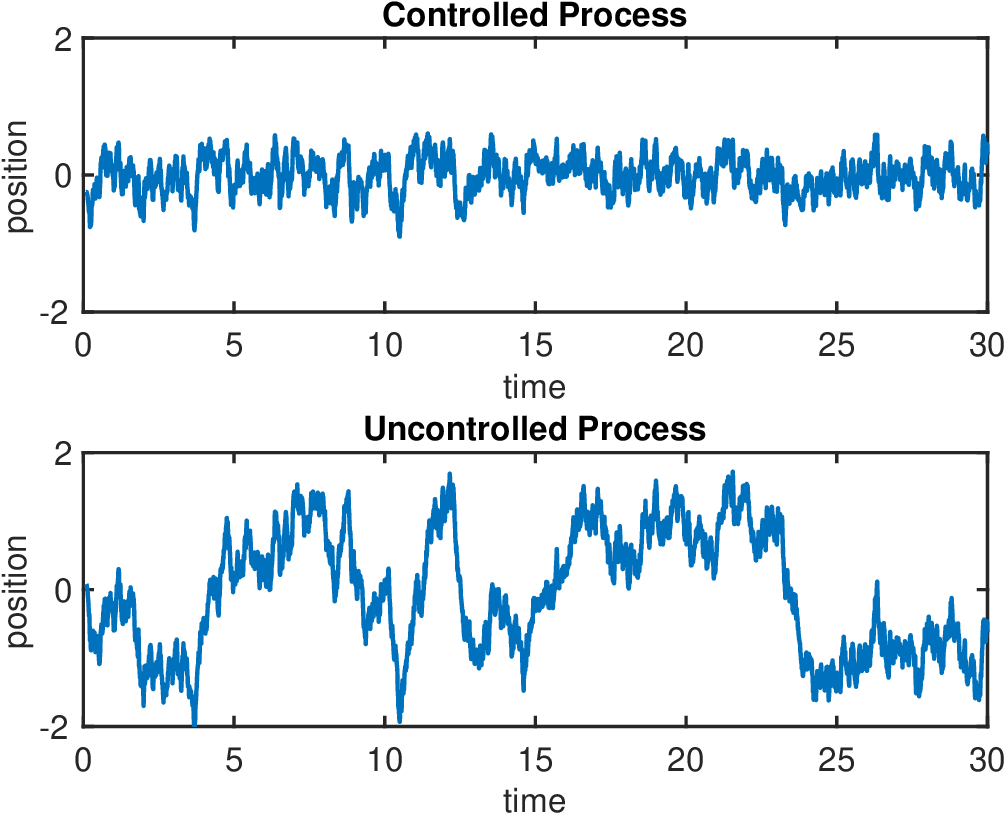} 
	\end{center}
	\caption{Comparison of the controlled and uncontrolled Langevin dynamics. Displayed is the time evolution of a single realisation of the SDE (\ref{eq:LD}) with and 
	without control.} \label{fig4}
\end{figure}

%
\section{Conclusions} \label{sec:conclusions}
%

Solving the HJB equation numerically constitutes a challenging task. Here we have provided a new perspective by 
combining forward and reverse evolution McKean--Vlasov equations with the tremendously successful EnKF methodology. We have done so by building on the previous work
\cite{MO22} and have generalized it to a wider class of forward and reverse McKean--Vlasov equations. In order to keep those
equations computationally tractable we have employed EnKF-type approximations to the McKean--Vlasov interaction terms. While not delivering
optimal control laws, the resulting time-dependent affine control laws can either be sufficient in themselves or, alternatively, may provide the starting point for
more accurate approximations such as the diffusion map approach outlined in Subsection \ref{sec:DM_EnKF}. 
We have applied the proposed methodology to a two-dimensional nonlinear control problem using only
$M=d_x + 1 = 3$ particles. Such a small ensemble size constitutes a significant improvements over the results presented in \cite{MO22} and \cite{JTMM22}. It remains
to be demonstrated that the methodology can be further extended to high-dimensional control problems in the spirit of EnKF applications to 
data assimilation, which deliver useful approximations even with $M \ll d_x$ ensemble members \cite{Evensenetal2022}. At the same time, strongly nonlinear
Langevin dynamics (\ref{eq:LD}) requires more sophisticated approximations of the grad-log terms in terms of diffusion maps. Still the ensemble size
could be kept at a moderate level ($M=8$) in order to 
recover the linear control law (\ref{eq:affine control}) robustly. 

\medskip

\noindent
{\bf Acknowledgment.} This work has been funded by Deutsche Forschungsgemeinschaft (DFG) - Project-ID 318763901 - SFB1294. The author thanks Manfred Opper for
insightful discussions on the subject of this work.

%
\section{Appendix} \label{sec:Appendix}
%

In this appendix, we discuss an extension of the proposed methodology to infinite horizon control problems with cost function
\begin{equation} \label{eq:infinte time}
J_\infty (x_0,u_{0:\infty}) = \mathbb{E} \left[ \int_0^\infty e^{-\gamma t} \left( c(X_t) + \frac{1}{2} \|u_t\|_R^2 \right) {\rm d}t \right].
\end{equation}
Here $\gamma \ge 0$ denotes the discount factor. The associated transformed HJB equation becomes
\begin{subequations} \label{eq:HJB infinity}
\begin{align}
-\partial_t v_t^\ast 
&= b \cdot \nabla_x  v_t^\ast + \frac{1}{2} \Sigma : D_x^2  v_t^\ast  \\
& \qquad -\,
\left(c + \gamma \log v_t^\ast + \frac{1}{2} \|\sigma^{\rm T} \nabla_x \log v_t^\ast ||^2  -  
\frac{1}{2} \|R G^{\rm T} \nabla_x \log v_t^\ast \|^2_R   \right) v_t^\ast 
\end{align}
\end{subequations}
with terminal condition $v_\infty^\ast = 1$.

The only modification to the finite horizon formulations from Section \ref{sec:mean field} concerns the drift function (\ref{eq:tilde f}), where $\tilde g_t(x)$ has to now satisfy the Poisson equation
\begin{equation} 
\nabla_x \cdot (\tilde \pi_t \tilde g_t) = -\tilde \pi_t \left( \gamma \log v_t^\ast + \frac{1}{2}\|\sigma^{\rm T} 
\nabla_x \log v_t^\ast \|^2 - \frac{1}{2} \|R G^{\rm T} \nabla_x \log v^\ast_t\|^2_R - \zeta_t\right)
\end{equation}
and the EnKF-based approximation of $\tilde g_t$ becomes
\begin{equation}
\tilde g_t^{\rm KF}(x) := \frac{1}{2} \tilde C_t \left\{ \gamma I + A_t \left(\Sigma(\tilde m_t^x)-G(\tilde m_t^x)RG(\tilde m_t^x)^{\rm T} \right)\right\} \left(A_tx + A_t\tilde m_t^x + 2c_t \right)
.
\end{equation}
Please note that $\tilde C_t A_t = \tilde C_t \bar C_t^{-1} - I < 0$ under the assumption that $\tilde C_t  < \bar C_t$. Hence, the additional drift term
stabilizes the time evolution of the reverse covariance matrix $\tilde C_t$ towards $\bar C_t$.

One expects the forward process (\ref{eq:determinstic McKean Vlasov}a) to reach an equilibrium distribution with mean $\bar m_{\rm eq}$ and covariance matrix $\bar C_{\rm eq}$ for $t>0$ sufficiently large. Furthermore, upon setting $\epsilon = 0$ in (\ref{eq:FB mean field}) and since we are in equilibrium, the time derivative ${\rm d}\bar X_t/{\rm d}t$ will either be zero or can be assumed to be relatively small. One may then fix these quantities in the reverse  process  (\ref{eq:determinstic McKean Vlasov}b)-(\ref{eq:determinstic McKean Vlasov}c), which, in turn, is integrated backward in time till an equilibrium distribution is reached with mean $\tilde m_{\rm eq}$ and covariance matrix $\tilde C_{\rm eq}$. The optimal control is provided by
\begin{equation}
u_t^\ast(x) = R G(x)^{\rm T} \left( \bar C_{\rm eq}^{-1}(x-\bar m^x_{\rm eq}) - \tilde C_{\rm eq}^{-1}(x-\tilde m_{\rm eq}^x) \right).
\end{equation}
Such a methodology provides an approximation to the stationary solution of the HJB equation (\ref{eq:HJB infinity}).

The combined EnKF and diffusion map approximation formulation from Section \ref{sec:DM_EnKF} can be generalized to the infinite horizon optimal control problem in a similar fashion. Please note that the numerical results from Subsection \ref{sec:CLD} already implied an essentially time-independent control law.

Alternatively, one can follow the actor-critic methodology to stochastic optimal control \cite{Meyn} and
introduce a family of control laws $u_\theta(x)$ parametrized by $\theta \in \mathbb{R}^{d_\theta}$ and set $\gamma = 0$ in (\ref{eq:infinte time}). For example, using a (stationary) linear control law of the form (\ref{eq:affine control}), the adjustable parameters, $\theta$, would be given by the (constant) matrix $A_t$ and the (constant) vector $c_t$.

More specifically, the actor chooses parameters, $\theta$,
and considers the controlled SDE 
\begin{equation}
    {\rm d}X_t = b(X_t){\rm d}t + G(X_t)u_\theta(X_t)
    {\rm d}t + \sigma(X_t) {\rm d}B_t.
\end{equation}
The generator of this SDE is denoted by $\mathcal{L}_\theta$. It is assumed that the SDE possesses an invariant density $\pi_\theta$; that is,
\begin{equation}
\mathcal{L}_\theta^\dagger \pi_\theta = 0,
\end{equation}
where $\mathcal{L}_\theta^\dagger$ denotes the adjoint of 
$\mathcal{L}_\theta$ \cite{Pavliotis2016}. The optimal $\theta_\ast$ is determined by
\begin{equation}
    \theta_\ast = \arg \min_\theta \pi_\theta[c_\theta]
\end{equation}
with cost
\begin{equation}
    c_\theta(x) = c(x) +  \frac{1}{2} \| 
    u_\theta(x)\|_R^2 .
\end{equation}
The critic provides the value function
$y_\theta$, which satisfies the stationary HJB equation
\begin{equation} \label{eq:ac_HJB}
    \mathcal{L}_\theta y_\theta + c_\theta - \pi_\theta[c_\theta] =0.
\end{equation}
Given the value function $y_\theta(x)$, the chosen parameter, $\theta$, can now be improved using the gradient \cite{Meyn}
\begin{equation} \label{eq:ac_gradient}
\nabla_\theta \pi_\theta [c_\theta]
= \pi_\theta \left[ (\nabla_\theta \mathcal{L}_\theta)
y_\theta \right] + \pi_\theta \left[ \nabla_\theta
c_\theta \right] = \pi_\theta \left[ \nabla_\theta u_{\theta}^{\rm T} \left( G^{\rm T} \nabla_x y_\theta + u_{\theta}  \right) \right] 
\end{equation}
and the optimal parameter value satisfies $\nabla_\theta \pi_{\theta_\ast}[c_{\theta_\ast}] =0$.

In order to extend our McKean--Vlasov approach to this control setting, we replace the stationary HJB equation (\ref{eq:ac_HJB}) by the forward-in-time HJB equation
\begin{equation}
\partial_t y_t = \mathcal{L}_\theta y_t + c_\theta - \pi_\theta [c_\theta]
\end{equation}
and apply the transformation $v_t(x) = \exp(-y_t(x))$ to obtain the modified HJB equation
\begin{equation}
\partial_t v_t = \mathcal{L}_\theta v_t - \left( 
c_\theta + \frac{1}{2}\|\sigma^{\rm T}\nabla_x \log v_t \|^2 - \zeta_t\right) v_t
\end{equation}
for $t\ge 0$ with initial condition $v_0(x) = 1$ and $\zeta_t$ an appropriate normalization constant. 

Adapting our previously developed methodology, we introduce the density
\begin{equation}
    \tilde \pi_t(x) := Z_t^{-1} v_{t}(x) \pi_\theta (x)
\end{equation}
with $Z_t = \pi_\theta [v_{t}]$ and find that
\begin{equation} \label{eq:ac_value}
    \nabla_x y_\theta = \nabla_x \log \pi_\theta - \lim_{t\to \infty} \nabla_x \log \tilde \pi_t .
\end{equation}
Furthermore, the density $\tilde \pi_t$ satisfies the forward evolution equation
\begin{subequations} \label{eq:ac_FPE}
\begin{align}
\partial_t \tilde \pi_t &= -\nabla_x \cdot (\tilde \pi_t 
\tilde b_{\theta}) + \frac{1}{2} \nabla_x \cdot ( 
\nabla_x \cdot (\tilde \pi_t \Sigma))
- \tilde \pi_t \left( c_\theta + \frac{1}{2}\|\sigma^{\rm T} \nabla_x \log v_{t} \|^2 - \tilde \zeta_{t}\right)\\
&= -\mathcal{L}_\theta \tilde \pi_t + 
\nabla_x \cdot(\tilde \pi_t \nabla_x \log v_t)
- \tilde \pi_t \left( c_\theta + \frac{1}{2}\|\sigma^{\rm T} \nabla_x \log v_{t} \|^2 - \tilde \zeta_{t}\right),
\end{align}
\end{subequations}
with modified drift function
\begin{equation}
    \tilde b_\theta (x) := 
    -b(x) - G(x) u_\theta (x) + \nabla_x \cdot \Sigma (x)
    + \Sigma (x) \nabla_x \log \pi_\theta (x)
\end{equation}
and normalisation constant $\tilde \zeta_t$.
The associated forward McKean--Vlasov evolution equation is for $\epsilon =0$ given by
\begin{equation}
    \frac{{\rm d}\tilde X_t}{{\rm d}t} = 
    \tilde b_{\theta} (\tilde X_t)
     - \tilde g_t(\tilde X_t) 
    - \frac{1}{2}
    \left( \nabla_x \cdot \Sigma(x) + \Sigma(x)
    \nabla_x \log \tilde \pi_t(x) \right) , 
\end{equation}
with initial $\tilde X_0 \sim \pi_\theta$ and 
and the McKean--Vlasov drift term $\tilde g_t(x)$ has to now satisfy
\begin{equation} \label{eq:ac_Poisson}
\nabla_x \cdot (\tilde \pi_t \tilde g_t) = 
- \tilde \pi_t\left(c_\theta +\frac{1}{2} \|\sigma^{\rm T} \nabla_x \log v_{t}\|^2 - \tilde \zeta_{t} \right).
\end{equation}
We may assume that (\ref{eq:ac_FPE}) possesses an invariant density $\tilde \pi_\theta$. Then (\ref{eq:ac_value}) reduces to
\begin{equation}
    \nabla_x y_\theta = \nabla_x \log \pi_\theta -  \nabla_x \log \tilde \pi_\theta .
\end{equation}

Using (\ref{eq:ac_value}), the chosen parameters can be improved via standard gradient descent based upon the gradient (\ref{eq:ac_gradient}) giving rise to time-dependent parameters, $\theta_t$, which, under suitable assumptions, converge to the optimal $\theta_\ast$. Furthermore, provided the parameters are adjusted slowly enough in time, one can make the assumption that $\mathcal{L}_{\theta_t} \pi_{\theta_t} \approx 0$. 
These assumptions suggest the coupled set of forward McKean--Vlasov evolution equations
\begin{subequations} \label{eq:ac_ODEs}
    \begin{align}
        \frac{{\rm d}X_t}{{\rm d}t} &= b(X_t) + G(X_t)u_{\theta_t}(X_t) 
  -\frac{1}{2}\nabla_x \cdot \Sigma (X_t) - \frac{1}{2}\Sigma(X_t) \nabla_x \log \pi_t(X_t) 
    ,\\
    \frac{{\rm d}\tilde X_t}{{\rm d}t}  &= 
    -b(\tilde X_t) - G(\tilde X_t) u_\theta (\tilde X_t) + \frac{1}{2} \nabla_x \cdot \Sigma (\tilde X_t)
    + \Sigma (\tilde X_t) \nabla_x \log \pi_t (\tilde X_t)
     - \tilde g_t(\tilde X_t) \\
     & \qquad \qquad 
    -\frac{1}{2}\Sigma(\tilde X_t) \nabla_x \log \tilde \pi_t(\tilde X_t) ,\\
    \frac{{\rm d} \theta_t}{{\rm d}t} &= -\delta \left(
    \pi_t \left[ (\nabla_\theta \mathcal{L}_{\theta_t})
y_t \right] + \pi_t \left[ \nabla_\theta
c_{\theta_t} \right] \right) =
-\delta \,\pi_t \left[ \nabla_\theta u_{\theta_t}^{\rm T} \left( G^{\rm T} \nabla_x y_t + u_{\theta_t}  \right) \right]
    \end{align}
\end{subequations}
where $\delta >0$ is sufficiently small, $\tilde g_t(x)$ satisfies (\ref{eq:ac_Poisson}), and
\begin{equation}  \label{eq:ac_y}
    \nabla_x y_t(x) := \nabla_x \log \pi_t(x) -  \nabla_x \log \tilde \pi_t(x) .
\end{equation}
Here $\pi_t$ denotes the law of $X_t$ and $\tilde \pi_t$ the law of $\tilde X_t$. The numerical approximations introduced in Section \ref{sec:numerical implementation} can now be applied to this system of McKean--Vlasov SDEs as well. 

In line with the previously stated (\ref{eq:determinstic McKean Vlasov}), we note that (\ref{eq:ac_ODEs}b) can be rewritten in the form
\begin{equation} \label{eq:ac_ODEb}
\frac{{\rm d}\tilde X_t}{{\rm d}t}  =   
- \tilde g_t(\tilde X_t)
-\frac{{\rm d}X_t}{{\rm d}t}(\tilde X_t)
 + \frac{1}{2}\Sigma(\tilde X_t) \nabla_x y_t (\tilde X_t) .
\end{equation}
In this context, it is worthwhile to consider the special case $G = R = I$, $\Sigma = 2I$, $b(x) = -\nabla_x U(x)$, and
$u_\theta(x) = -\nabla_x \Psi_\theta(x)$ in more detail. Here $U(x):\mathbb{R}^{d_x}\to \mathbb{R}$ 
and $\Psi_\theta(x): \mathbb{R}^{d_x} \to \mathbb{R}$ 
are given functions. Under these assumptions, the density $\pi_\theta$ is explicitly known and 
\begin{equation}
    \nabla_x \log \pi_\theta(x) = 
    -\nabla_x U(x) - \nabla_x \Psi_\theta(x).
\end{equation}
Furthermore, we may assume that (\ref{eq:ac_ODEs}a) is in equilibrium and we can set
\begin{equation}
\frac{{\rm d}X_t}{{\rm d}t} \equiv 0
\end{equation}
in (\ref{eq:ac_ODEb}). Let $\tilde g_t(x) = \nabla_x V_t(x)$ denote the solution of (\ref{eq:ac_Poisson}) for appropriate potential $V_t(x)$, then it follows from (\ref{eq:ac_ODEb}) and ${\rm d}\tilde X_t/{\rm d}t \approx 0$ that
\begin{equation}
\nabla_x y_t(x) \approx \nabla_x V_t(x).
\end{equation}
Hence,
\begin{equation}
    \pi_t \left[ (\nabla_\theta \mathcal{L}_{\theta_t})
y_t \right] + \pi_t \left[ \nabla_\theta
c_{\theta_t} \right] \approx
\pi_{\theta_t} \left[ 
\nabla_\theta (\nabla_x \Psi_{\theta_t})^{\rm T} 
\left(\nabla_x V_t + \nabla_x \Psi_{\theta_t}\right)
\right]
\end{equation}
and the optimal parameter choice, $\theta_\ast$, satisfies
\begin{equation} \label{eq:ac_c1}
    0 =
\pi_{\theta_\ast} \left[ 
\nabla_\theta (\nabla_x \Psi_{\theta_\ast})^{\rm T} 
\left( \nabla_x V_{\theta_\ast} + \nabla_x \Psi_{\theta_\ast}\right)
\right]
\end{equation}
subject to the potential $V_{\theta_\ast}(x)$ satisfying the Poisson equation
\begin{equation} \label{eq:ac_c2}
\nabla_x \cdot(\tilde \pi_{\theta_\ast} \nabla_x V_{\theta_\ast})
= - \tilde \pi_{\theta_\ast} \left(c_{\theta_\ast} + \|\nabla_x V_{\theta_\ast}\|^2 - \zeta_\ast\right)
\end{equation}
with $\tilde \pi_{\theta_\ast} \propto e^{-V_{\theta_\ast}} \pi_{\theta_\ast}$.
The approach proposed in this paper can now be viewed as providing a dynamic particle-based algorithm for solving the nonlinear equations (\ref{eq:ac_c1})-(\ref{eq:ac_c2}). It is also worth noting that (\ref{eq:ac_c2}) is equivalent to
\begin{equation}
    \mathcal{L}_{\theta_\ast} V_{\theta_\ast} =
    - c_{\theta_\ast} + \pi_{\theta_\ast}[c_{\theta_\ast}],
\end{equation}
which implies $\nabla_x y_{\theta_\ast} = \nabla_x V_{\theta_\ast}$ as desired.


\bibliographystyle{plainurl}
%
\bibliography{bib-database}
%

%
%


\end{document}